\documentclass[12pt]{amsart}
\usepackage{epsfig}  	
\usepackage{amsfonts,amssymb}
\usepackage{amsmath}
\usepackage{amsthm}
\usepackage{graphicx,color}   

\newtheorem{theorem}{Theorem}

\newtheorem{lemma}{Lemma}

\begin{document}
\begin{center}
  \large 
  Inverse extremal problem for an anti-tumor \\therapy model \par \bigskip
  \normalsize
Andrey Kovtanyuk\textsuperscript{1,$\star$}, Christina Kuttler\textsuperscript{1}, Kristina Koshel\textsuperscript{2}, \\ and Alexander Chebotarev\textsuperscript{2} 
\par \bigskip

\textsuperscript{1}\,Technical University of Munich, Mathematical Department, Boltzmannstr. 3, 85748, Garching, Germany \par
\textsuperscript{2}\,Far Eastern Federal University, AjaxBay 10, Russky Island, \\ 690922, Vladivostok, Russia \par \bigskip

email: \textsuperscript{$\star$}\,kovtanyu@ma.tum.de

\end{center}

\begin{abstract}
An optimal control problem for a model of tumor growth is studied. In a given subdomain, it is required to minimize the density of tumor cells, while the drug concentration in tissue is limited by given minimal and maximal values. Based on derived estimates of the solution of the controlled system, the solvability of the control problem is proved. The problem is reduced to an optimal control problem with a penalty. An algorithm for solving the optimal control problem with a penalty is constructed and implemented. The efficiency of the algorithm is illustrated by a numerical example.
\end{abstract}


\section{Introduction}

Currently, there are many mathematical models of tumor evolution based on ODEs or PDEs  \cite{Schonfeld, Wagner, Fritz, Kuznetsov}. Often such models are quite complex and include several state variables. For example, the article \cite{Wagner} considers a model that contains five state variables, two of which are volume fractions of proliferative and necrotic tumor cells. A 3D--1D model studied in \cite{Fritz} includes seven state variables of which three relate to volume fractions of different species of tumor cells (proliferative, hypoxic, and necrotic cells). A model considered in \cite{Kuznetsov} contains eleven state variables including three species of tumor cells. Note that the models from \cite{Wagner, Fritz, Kuznetsov} also take into account the influence of anti-tumor drug therapy. 

Some studies of tumor evolution are based on reaction-diffusion models in which tumor cells are not divided into species (see, e.g., \cite{Harpold, Alfonso, Colli, Viguerie}, where tumor cell growth simulation results were presented). Many diffusion models on tumor cell growth and death contain nonlinear reaction terms. Note that the equation for glioma cell growth in \cite{Alfonso}, in addition to the nonlinear reaction term, also contains a nonlinear diffusion term. 

Optimal control of the volume fraction of tumor cells under the influence of drug therapy is important for assessing the effectiveness of the treatment plan and its correction. There are a series of works on the optimal control of tumor treatment that do not take into account its spatial distribution (see, e.g., \cite{Schattler, Badri}). Optimal control problems for diffusion models of tumor treatment are studied in \cite{Colli, Sowndarrajan}.

In the current paper, a model of tumor growth under the influence of drug therapy is under consideration. The model describes the behavior of two state variables: the normalized density of tumor cells and normalized drug concentration. An optimal control problem for the considered model is studied. Following this, in a given subdomain it is necessary to minimize the density of tumor cells, while the drug concentration is limited to the specified minimum and maximum values. The solvability of the optimal control problem is proved. An algorithm for solving the optimal control problem by minimizing the objective functional with a penalty is proposed. The efficiency of the algorithm is illustrated by a numerical example.

\section{Formalization of the optimal control problem}

We consider the following initial-boundary value problem describing the growth and death of tumor cells in the domain $\Omega$ with boundary $\Gamma$ during time interval $(0,T)$:
\begin{equation}
\label{eq21}
y' - \nabla \cdot (k \nabla y) = d(s)y,\;\;\; y|_\Gamma = 0, \;\;\; y(x,0) = y_0;
\end{equation}
\begin{equation}
\label{eq1}
s'(t) + M_0 s(t) = u(t),\;\;\; s(0) = 0,
\end{equation}
where $y$ is the normalized density of tumor cells, $s$ is the normalized drug concentration in tissue, $u$ describes the drug supply, $k = k(x)$ is the diffusion coefficient, $M_0$ is positive constants.

Let $\Omega$ be a Lipschitz bounded domain, $Q=\Omega\times (0,T)$, $\Sigma=\Gamma\times (0,T)$.
We denote by $L^p$, $1 \leq p \leq \infty$, the Lebesgue space, by $H^1$ the Sobolev space $W^1_2$, and by $L^p(0,T;X)$ the Lebesgue space of functions from $L^p$, defined on $(0, T)$, with values in Banach space $X$.
Also let $H = L^2(\Omega), \; V=H^1_0(\Omega)=\{v\in H^1(\Omega):\; v|_{\Gamma}=0\}$, and the space $V'$ be the dual to $V$. Then we identify $H$ with its dual space $H'$ such that  $V \subset H^1(\Omega)\subset H = H' \subset (H^1(\Omega))'\subset V'$, and denote by $\|\cdot\|$ the norm in $H$, and by $(h,v)$ the value of functional $h\in V'$ on the element $v\in V$ coinciding with the inner product in $H$ if $h \in H$.
We also define the control space $U=L^2(0,T)$ and the state space 
$$Y = \{y \in L^2(0,T;V) \colon y'= dy/dt\in L ^2(0,T,V') \}.$$

We assume that the input data satisfies the following conditions:
\vskip 0.2cm
\noindent $(i) \;\;\; 0 < k_0 \leq k(x) \leq k_1,\;\; k_j=Const$,\;\; $y_0\in H$;
\vskip 0.2cm
\noindent $(ii) \;\; d(s)(s-s_m)<0,\;\; s\neq s_m$, 
\vskip 0.1cm
\;\;\, $|d(s_1)-d(s_2)|\leq L_a|s_1-s_2|\;\; \forall \;|s_{1,2}|\leq a.$
\vskip 0.2cm

Let us define the operator $A\colon V \to V'$ using the equality: 
$$\forall\; y,v\in V\quad (Ay, v) = (k\nabla y, \nabla v). $$
Therefore, for $y \in Y$, the problem (\ref{eq21}) can be rewritten in the following weak form:
\begin{equation}
\label{eq2}
y' + Ay = d(s)y,\;\;\; y(0) = y_0.
\end{equation}
Here and further, when writing differential equations, we assume that they hold over the entire time interval $(0,T)$.

{\bf Problem ($P$)} For $\lambda > 0$, it is required to minimize the objective functional 
\begin{equation}
\label{eq0}
J(y,u)= \frac{1}{2}\|y\|^2_{L^2(Q)} + \frac{\lambda}{2}\|u\|^2_{L^ 2(0,T)}\rightarrow \inf 
\end{equation}
on functions $y\in Y$, $u\in U$ satisfying the conditions (\ref{eq1}) and (\ref{eq2}) such that 
\[
s(t)\leq s_+,\;\; t\in (0,T); \;\; s(t)\geq s_-,\;\; t\in (t_0,T), 
\]
where $s_-$, $t_0$ are positive constants.

\section{Solvability of the optimal control problem}

First, consider the properties of problems \eqref{eq1} and \eqref{eq2}. Define the operator $B: L^2(0,T)\to H^1(0,T)$ so that $s=B(u)$ if $s$ is a solution of problem \eqref{eq1}.
\begin{lemma}
For $s=B(u)$ the following estimates hold:
\begin{equation}
\label{B1}
|s(t)| \leq \sqrt{t}\|u\|_U,\;\;\;  \|s\|_U\leq T\|u\|_U,\;\;\; \|s'\|_U\leq (1+M_0T)\|u\|_U.
\end{equation}
\end{lemma}
\begin{proof}
The first two estimates \eqref{B1} immediately follow from the representation of the solution of problem \eqref{eq1}, 
$$s(t)=\int_0^te^{-M_0(t-\tau)}u(\tau)d\tau.$$ 
The inequality for $\|s'\|_U$ follows from the equation \eqref{eq1} and the estimate for $\|s\|_U$.
\end{proof}

\begin{lemma}
Let conditions (i),\,(ii) hold. Then for $u\in U$ there is a unique solution $y\in Y$ of problem \eqref{eq2}, where $s=B(u)$.
\end{lemma}
\begin{proof}
Let $\chi=d(s)$. By Lemma 1, $|s|\leq a$, where $a = \sqrt{T}\|u\|_U$, and therefore $|\chi|\leq |d(0)|+|d(s )-d(0)|\leq |d(0)|+ L_a a = a_1$. Problem \eqref{eq2} can be written as 
$$
y'+A_1y=0,\;\;\; y(0)=y_0,
$$
where $A_1y=Ay-\chi y$. Note that 
$$(A_1z,z)+a_1\|z\|^2\geq (k\nabla z, \nabla z)\geq k_0\|z\|^2_V. $$ 
Therefore \cite[Ch.3, Th.1.2]{Lions} there is a unique solution $y\in Y$ of problem \eqref{eq2}.
\end{proof}

\begin{theorem}
Let conditions (i),\,(ii) hold and
\begin{equation}\label{S}
\frac{1-e^{-M_0T}}{1-e^{-M_0t_0}}s_-\leq s_+.
\end{equation}
Then there is a solution of problem (P).
\end{theorem}

\begin{proof}
Let us show that the set of controls ensuring the fulfillment of the inequalities $s(t)\leq s_+,\; t\in [0,T]$, $s(t)\geq s_-,\; t\in [t_0,T]$, is non-empty. Indeed, if we take $u:=M_0(1-e^{-M_0t_0})^{-1}s_-$ as a control, then the solution to problem \eqref{eq1} has the form $s(t)=u(1 -e^{-M_0t})/M_0$ and then $s(t)\geq s_-,\; t\in [t_0,T]$. Condition \eqref{S} guarantees that $s(t)\leq s_+,\; t\in [0,T]$.

Let $J(y_m,u_m)\to j=\inf J$,
\begin{equation}
\label{E}
y_m'+Ay_m=d(s_m)y_m, \;\;\; s_m=B(u_m),\;\;\; y_m(0)=y_0;
\end{equation}
$$
s_m(t)\leq s_+,\;\;\; t\in [0,T],\;\;\; s_m(t)\geq s_-,\;\;\; t\in [t_0,T].
$$

The structure of $J$ implies the estimate $\|u_m\|_U\leq C$ and therefore, by Lemma 1,
$$
|s_m(t)| \leq C,\;\;\;   \|s_m\|_{H^1(0,T)}\leq C,\;\;\; |d(s_m(t))| \leq C.
$$
Here and below, by $C>0$, we denote constants independent of $m$. Multiplying in the sense of the inner product the equation for $y_m$ in \eqref{E} by $y_m$ and taking into account the estimates of $s_m$, we derive the inequality
$$
\frac{1}{2}\frac{d}{dt}\|y_m\|^2+k_0\|\nabla y_m\|^2\leq C\|y_m\|^2.
$$
Integrating this inequality over $t$ and applying Gronwall's lemma, we obtain the estimate $\|y_m\|_Y\leq C.$ Based on the estimates obtained, we conclude, passing to subsequences if necessary, that there exist functions $u\in U,\; s\in H^1(0,T),\; y\in Y$ such that
\begin{equation*}
u_m\to u\, \text{ weakly in }U,\;\;\; s_m\to s\, \text{ weakly in }H^1(0,T), \text{ strongly in }U,
\end{equation*}
\begin{equation*}
y_m\to y\, \text{ weakly in }L^2(0,T;V), \;\; \text{ strongly in }L^2(0,T;H).
\end{equation*}
Moreover, the Lipschitz property of $d$ guarantees that
$d(s_m)\to d(s)$ in $U=L^2(0,T)$. The obtained results on convergences allow us to pass to the limit in \eqref{E}. As a result,
$$
y'+Ay=d(s)y,\;\;\; s=B(u),\;\;\; y(0)=y_0;
$$
$$
s(t)\leq s_+,\;\; t\in [0,T];\;\;\; s(t)\geq s_-,\;\; t\in [t_0,T].
$$
Since the objective functional is weakly lower semicontinuous, then
$
j\leq J(y,u)\leq \liminf J(y_m,u_m)=j
$
and therefore the pair $\{y,u\}$ is a solution of problem ($P$).
\end{proof}

\section{Problem with penalty}

Let us define the constraint operator
\begin{equation*}
F: W\times H^1(0,T)\times U\to L^2(0,T;V')\times U\times H\times \mathbb R
\end{equation*}
such that
\begin{equation*}
F(y,s,u) =  \{y'+Ay-d(s)y, s'+ M_0 s - u, y(0)-y_0, s(0)\}.
\end{equation*}
\vskip 0.2cm

Further, we consider the following extremal problem with a penalty.\vskip 0.2cm

\textbf{Problem ($P_\varepsilon$})
\begin{multline*}
J_\varepsilon(y,s,u) =  J(y,u)
+\frac{1}{\varepsilon} \int\limits_0^T f_1(s(t)) dt
+\frac{1}{\varepsilon} \int\limits_{t_0}^T f_2(s(t)) dt
 \rightarrow \inf,\\ F(y,s,u)=0,\;\; u\in U.
\end{multline*}
Here, $\varepsilon>0$,
\[
f_1(s) =
\begin{cases}
0, & \text{if $s\leq s_+$,} \\
 (s - s_+)^2, & \text{if $s > s_+$,}
\end{cases}
\]
\[
f_2(s) =
\begin{cases}
0, & \text{if $s\geq s_-$,} \\
 (s - s_-)^2, & \text{if $s < s_-$.}
\end{cases}
\]

Estimates for the solution of the controllable system make it possible to prove the solvability of the problem with a penalty ($P_\varepsilon$) similar to the proof of Theorem 1.

\begin{theorem}
Let conditions (i),\,(ii) hold.
Then a solution of the problem ($P_\varepsilon$) exists.
\end{theorem}

Consider the approximation properties of solutions of the problem with a penalty.

\begin{theorem}
Let conditions (i),\,(ii) hold.
If $\{y_\varepsilon,s_\varepsilon,u_\varepsilon\}$ is a solution of the problem  ($P_\varepsilon$)  for $\varepsilon>0$, then there is a sequence $\varepsilon\to +0$ such that
\begin{multline*}
u_\varepsilon \rightarrow \widehat u\; \text{ weakly in }U,\;\;\; y_\varepsilon\rightarrow \widehat y\;
\text{ strongly in }L^{2}(0,T;H),\\ s_\varepsilon \rightarrow \widehat s\; \text{ strongly in }U,
\end{multline*}
where $\{\widehat y,\,\widehat u\}$ is a solution of Problem ($P$), $\widehat s = B(\widehat u)$.
\end{theorem}
\begin{proof}
Consider the pair $\{y,u\}$, a solution to problem ($P$), the existence of which is guaranteed by Theorem 1. Since $s=B(u)$ satisfies the inequalities $s(t)\leq s_+,\; t\in [0,T]$; $s(t)\geq s_-,\; t\in [t_0,T]$ that means
$f_1(s(t))=0,\; t\in [0,T]$, $f_2(s(t))=0,\; t\in [t_0,T]$, then
\begin{equation*}
J(y_\varepsilon,u_\varepsilon)+\frac{1}{\varepsilon} \int\limits_0^T f_1(s_\varepsilon(t)) dt
+\frac{1}{\varepsilon} \int\limits_{t_0}^T f_2(s_\varepsilon(t)) dt \\
\leq J(y,u).    
\end{equation*}
Here, $s_\varepsilon=B(u_\varepsilon)$. Therefore 
\begin{equation}
\label{P}
J(y_\varepsilon,u_\varepsilon)\leq J(y,u),\;\;\; \int\limits_0^T f_1(s_\varepsilon(t)) dt
+\int\limits_{t_0}^T f_2(s_\varepsilon(t)) dt
\leq \varepsilon J(y,u).
\end{equation}
From the first inequality \eqref{P} it follows that $\|u_\varepsilon\|_U\leq C$ and similarly to the proof of Theorem 1 we derive the estimates
\[
|s_\varepsilon(t)| \leq C,\;\;\; \|s_\varepsilon\|_{H^1(0,T)}\leq C,\;\;\;
|d(s_\varepsilon(t))| \leq C, \;\;\; \|y_\varepsilon\|_Y\leq C.
\]
Here, by $C>0$, we denote constants independent of $\varepsilon$. Based on the estimates obtained, we conclude, passing to subsequences if necessary, that there are functions $\widehat u\in U,\; \widehat s\in H^1(0,T),\; \widehat y\in Y$ such that
\begin{equation*}
u_\varepsilon\to \widehat u\;\, \text{ weakly in }U,\;\; s_\varepsilon\to \widehat s\;\, \text{ weakly in }H^1(0,T), \text{ strongly in }U,
\end{equation*}
\begin{equation*}
y_\varepsilon\to \widehat y\;\, \text{ weakly in }L^2(0,T;V),\;\, \text{ strongly in }L^2(0,T;H).
\end{equation*}
At the same time $d(s_\varepsilon)\to d(\widehat s)$ in $U=L^2(0,T)$.
The obtained results on convergence allow us, by passing to the limit in the equality $F(y_\varepsilon,s_\varepsilon,u_\varepsilon)=0$, to arrive at the relation $F(\widehat y,\widehat s,\widehat u)=0$. Moreover,
$$
\int\limits_0^T f_1(s_\varepsilon(t)) dt\to \int\limits_0^T f_1(\widehat s(t)) dt,\;\;\;
\int\limits_{t_0}^T f_2(s_\varepsilon(t)) dt\to \int\limits_{t_0}^T f_2(\widehat s(t)) dt,
$$
and since by virtue of \eqref{P} these limits equal to $0$, then
$\widehat s(t)\leq s_+,\; t\in [0,T]$;\; $\widehat s(t)\geq s_-,\; t\in [t_0,T].$
Therefore, the pair $\{\widehat y,\widehat u\}$ is admissible for problem ($P$). Notice that 
$$
j=\inf J\leq J(\widehat y,\widehat u)\leq \liminf J(y_\varepsilon,u_\varepsilon)\leq J(y,u)=j,
$$
and therefore $\{\widehat y,\widehat u\}$ is a solution of problem ($P$).
\end{proof}

To obtain the optimality system, we will use the Lagrange principle for smooth-convex extremal problems~\cite{Furs}. Let us establish the non-degeneracy of the optimality conditions by showing that the image of the derivative of the constraint operator with respect to $s$ coincides with the space $W = L^2(0,T;V')\times U\times H\times \mathbb R$.
\begin{lemma}
Let conditions (i),\,(ii) hold. For any $\widehat y \in Y$, $\widehat s\in H^1(0,T)$, $\widehat u \in U$ the following equality is true:
$$
\texttt{Im}\,F'_\nu (\widehat y,\widehat s,\widehat u) = W.
$$
Here, $\nu=\{y,s\}$.
\end{lemma}
\begin{proof}
Consider an arbitrary element $\eta= \{\zeta,q,z_0,h_0\}\in W$
and check that there is a solution $z\in Y$, $h\in H^1(0,T)$ of the system
\[
z'+Az-d(\widehat s)z=\zeta+d'(\widehat s)h\widehat y,\;\;z(0)=z_0,
\] 
\[
h' + M_0 h = q,\;\; h(0)=h_0.
\]
The existence of a solution $h\in H^1(0,T)$ of the Cauchy problem for the linear ordinary differential equation is obvious.
The solvability of the parabolic problem for $z$ with the right-hand side $\zeta + d'(\widehat s)h\widehat y\in L^2(0,T;V')$ follows from \cite[Ch.3, Th. 1.2]{Lions}.
\end{proof}

\begin{theorem}
Let conditions (i),\,(ii) hold.
If $\{y_\varepsilon,s_\varepsilon,u_\varepsilon\}$ is a solution of  problem  ($P_\varepsilon$)  for $\varepsilon>0$, then there is a unique solution $\{p_1,p_2\}\in W\times H^1(0,T)$ of the dual system
\begin{equation}\label{D1}
-p_1' + Ap_1 - d(s_\varepsilon)p_1 = -y_\varepsilon,\;\;\;  p_1(T)=0;
\end{equation}
\begin{multline}\label{D2}
-p_2' + M_0p_2 + \frac{1}{\varepsilon}f_1'(s_\varepsilon) + \frac{1}{\varepsilon}\chi(s_\varepsilon,t)\\ = d'(s_\varepsilon)(y_\varepsilon,p_1),\;\;\; t\in (0,T),\;\;\; p_2(T)=0,
\end{multline}
where $\chi(s,t)=f_2'(s)$ if $t\in [t_0,T]$ and $\chi(s,t)=0$ if $t\in [0,t_0)$.
Moreover, the parameter $u_\varepsilon$ is defined by the following equality $u_\varepsilon=\lambda^{-1}p_2.$ 
\end{theorem}
\begin{proof}
The statement of the theorem follows from \cite[Chapter 2, Theorem 1.5]{Furs}.
By Lemma 3, the Lagrange function of problem ($P$) is defined by the equality
\begin{multline*}
L(y,s,u,p_1,p_2,q_1,q_2)= J_\varepsilon(y,s,u)+\int_0^T(y'+Ay-d(s)y,p_1)dt+ \\
\int_0^T(s'+M_0s-u)p_2dt + (y(0)-y_0,q_1)+s(0)q_2,
\end{multline*}
where\, $\{p_1,p_2\}\in Y\times H^1(0,T),\;\, q_1\in H,\;\, \text{and}$ $q_2\in \mathbb R$. In accordance with Lagrange's principle, equalities
$L_y'(y_\varepsilon,s_\varepsilon,u_\varepsilon,p_1,p_2,q_1,q_2)=0$ and $L_s'(y_\varepsilon,s_\varepsilon,u_\varepsilon,p_1,p_2, q_1,q_2)=0$
give the conjugate system \eqref{D1}, \eqref{D2}, and the condition $L_u'(y_\varepsilon,s_\varepsilon,u_\varepsilon,p_1,p_2,q_1,q_2)=0$ implies the equality $u_\varepsilon=\lambda^{-1} p_2.$
\end{proof}

\section{Numerical algorithm}

\noindent 1. Choose the value of the gradient step $\delta$, the number of iterations $N$, the value of the penalty coefficient $\varepsilon$, and the initial approximation for the control $u_0 \in U$.
\vskip 0.2cm

\noindent 2.  Loop for $k = 0, 1, 2,\dots, N$. \\
For a given $u_k$, calculate the state $y_k,\,s_k$ that is the solution of the problem \eqref{eq1},\,\eqref{eq2}.
\vskip 0.2cm

\noindent 3. Calculate the value of the objective functional $J_\varepsilon(y_k, s_k, u_k)$.
\vskip 0.2cm

\noindent 4. Calculate the conjugate state $\{p_{1k},\,p_{2k}\}$ from the equations \eqref{D1},\,\eqref{D2} setting $y_\varepsilon=y_k,\; s_\varepsilon=s_k$.
\vskip 0.2cm

\noindent 5. 
Recalculate the control $u_{k+1} = u_k - \delta (\lambda u_k - p_{2k})$.
\vskip 0.2cm

The value of the parameter $\delta$ is chosen empirically in such a way that the value $\delta (\lambda u_k - p_{2k})$ is a significant correction for $u_{k+1}$. The number of iterations $N$ is chosen sufficient to satisfy the condition
$J_\varepsilon(y_k, s_k,u_k) - J_\varepsilon(y_{k+1}, s_{k+1}u_{k+1}) < \epsilon$, where $\epsilon>0$ determines the accuracy of calculations.

\section{Numerical example}

As the computational domain, we consider a square with an edge of 3\,cm. The tumor occupies the circle located in the center of the computational domain with a diameter of 1\,cm.  We set the following values of the problem parameters: $s_c = 0.2$, $s_- = 0.4$, $s_+ =0.8$, $T=28$\;(days). The value of the diffusion coefficient $k$ is  $2.5\cdot 10^{-9}$\;(cm$^2$/s) \cite{Kolobov}. The initial distribution of the control is set as follows: it equals 0.00014 (s$^{-1}$) during the first hour of each day. The value of the penalty coefficient $\varepsilon$ is set equal to 0.2. The number of iterations is 10.

The implementation of the computational algorithm was fulfilled by the finite element method package FreeFEM\,++ \cite{Hecht}. In Figure~1, the behavior of the objective functional $J_\varepsilon$ in dependence on the number of iterations is shown.

In Figure~2, the behavior of the control $u_\varepsilon$ during 28 days of observation is shown. The dots indicate the amount of medication taken every 8 hours. The red plot is their continuous approximation. In Figure~3, the behavior of the drug concentration in tissue $s_\varepsilon$ (blue plot) in comparison with the minimal and maximal constraints ($s_-$ and $s_+$, red lines) and the critical level of drug concentration ($s_c$, violet line) is shown. In Figure~4, the distributions of tumor cell density $y_\varepsilon$ in the central cross-section of the computational domain for different time moments (initial time moment, 1 week, 2 weeks, 3 weeks, and 4 weeks) are shown. As can be seen, during the whole time interval, a monotonic decrease in tumor cell density is observed. 

\section{Conclusion}

The optimal control problem for the model of tumor evolution under the influence of drug therapy has been studied. The solvability of the problem has been proved. The numerical algorithm for finding a solution has been constructed, its convergence has been established. 

Subsequent efforts of the authors will be aimed at studying a model of tumor evolution that includes a nonlinear reaction term and describes the behavior of two or more species of tumor cells.

\centerline{
\includegraphics[width=115mm]{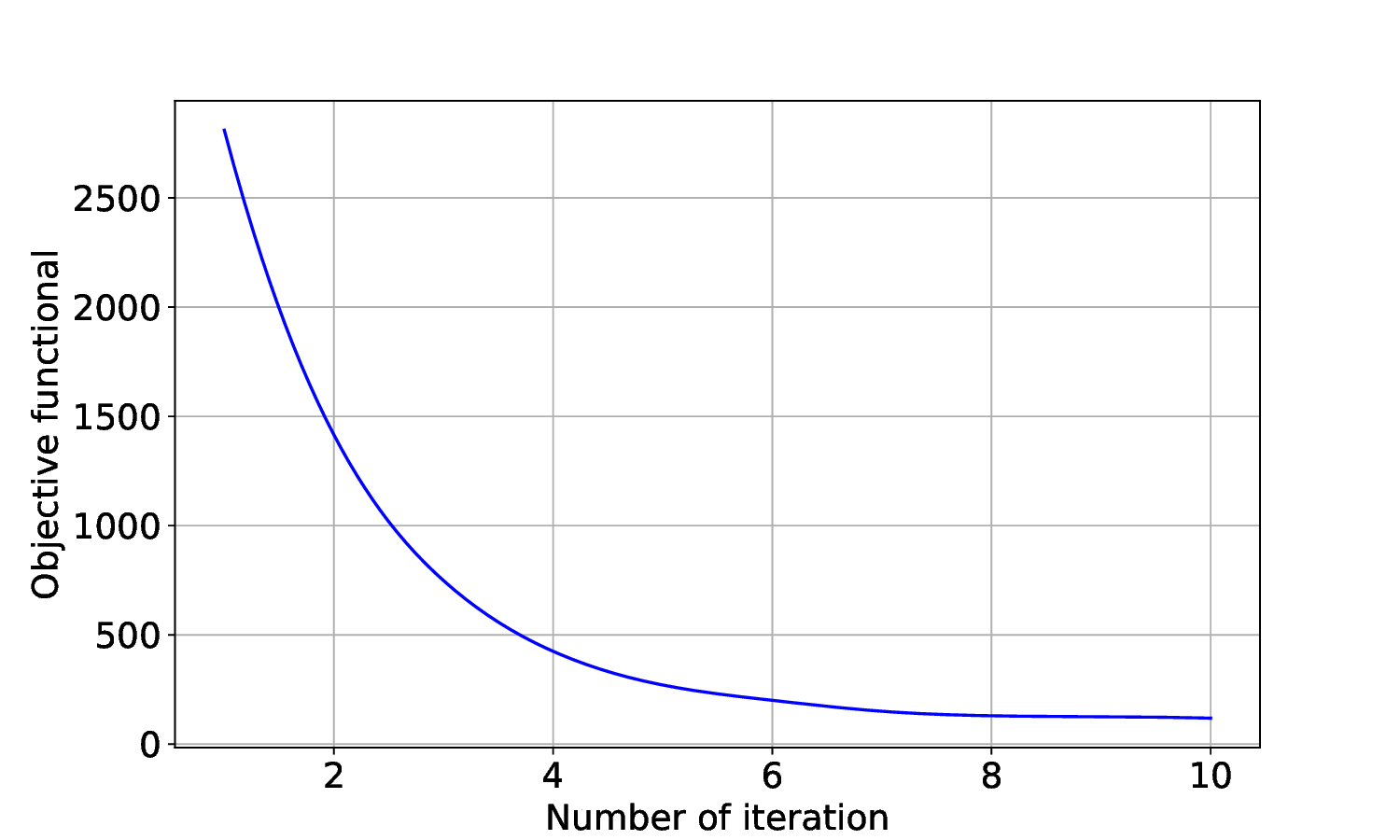}}
\medskip
\vskip -0.1cm
\centerline{\narrower \small \textbf{Figure~1:} Objective functional. }

\centerline{
\includegraphics[width=115mm]{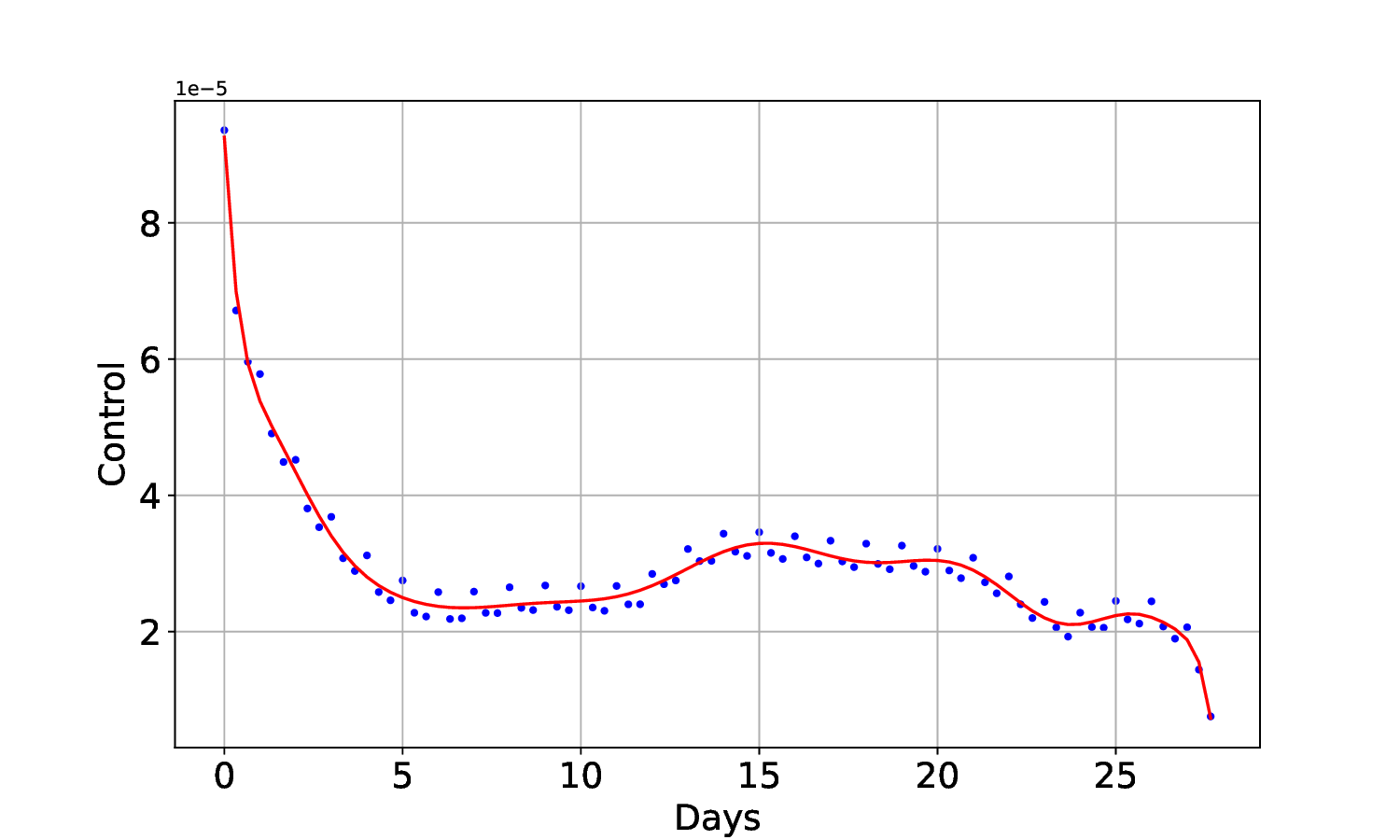}
}
\vskip -0.1cm
{\narrower\noindent \small \textbf{Figure~2:} Control (dots) and its continuous approximation (red plot). \par}

\centerline{
\includegraphics[width=115mm]{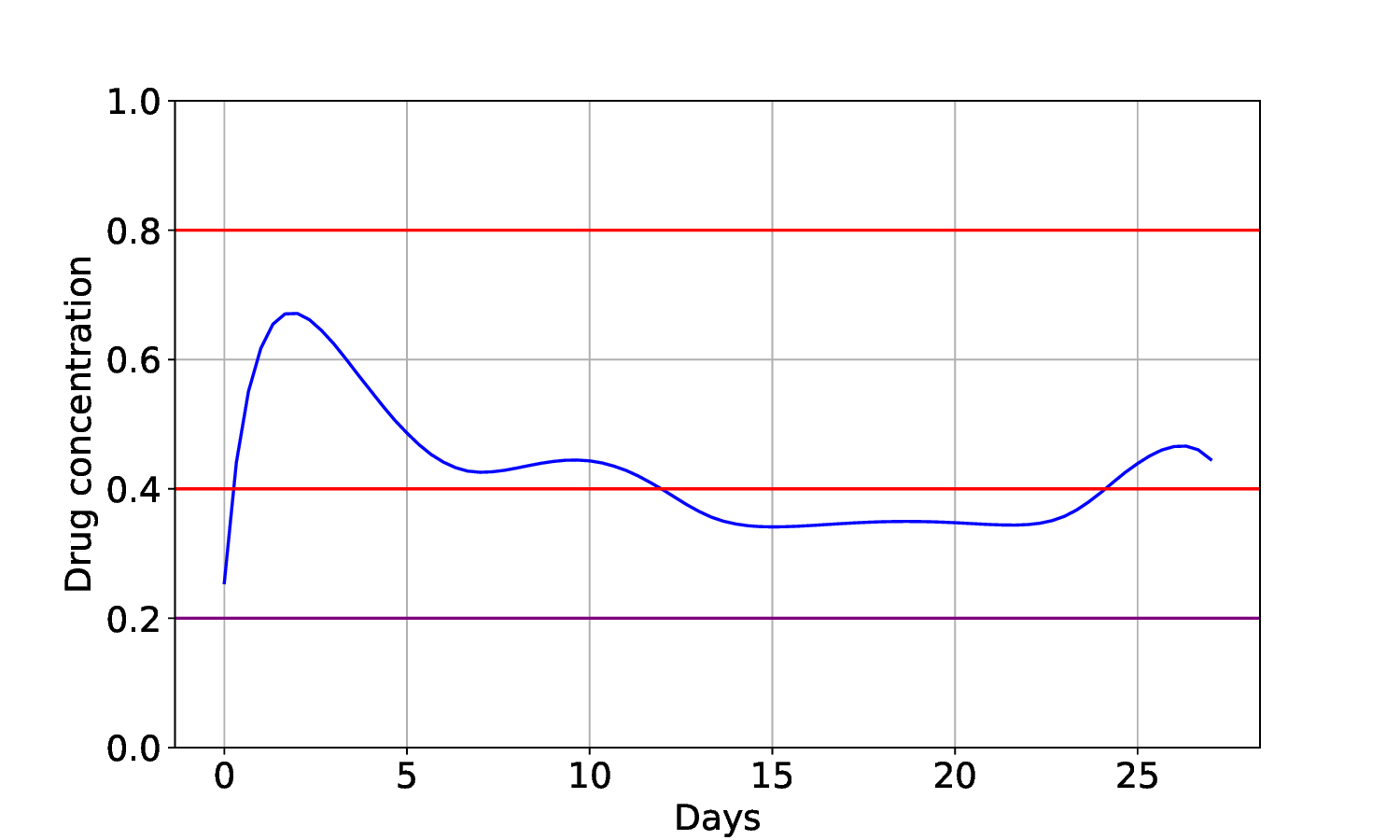}
}
\vskip 0.2cm
{\narrower\noindent\small \textbf{Figure~3:} Drug concentration (blue plot), constraints on drug concentration (red lines), and the critical value of drug concentration (violet line). \par}
\vskip 0.2cm
\centerline{
\includegraphics[width=115mm]{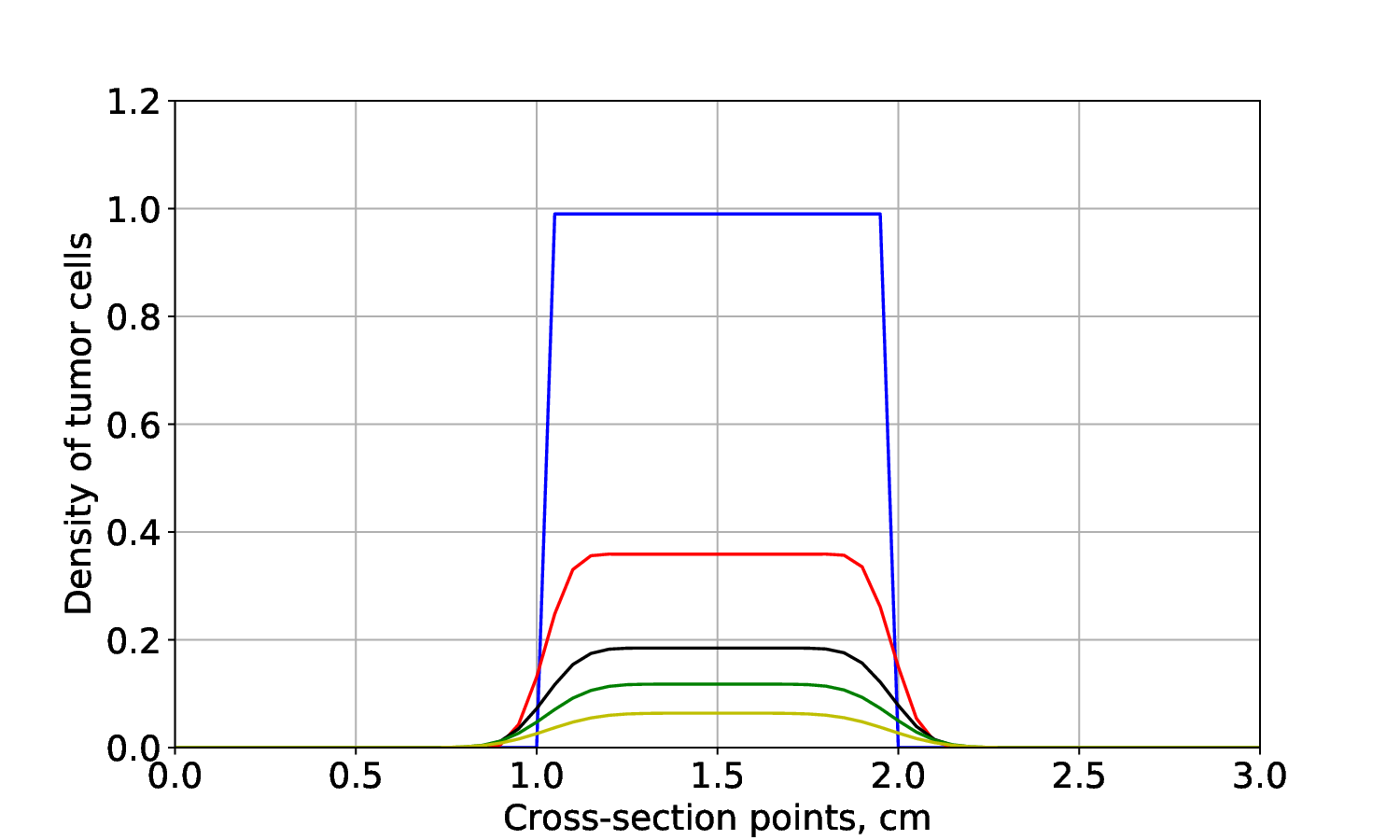}
}
\vskip 0.2cm
{\narrower\noindent\small \textbf{Figure~4:} Density of tumor cells in different time moments: initial distribution (blue), 1 week (red), 2 weeks (black),  3 weeks (green), and 4 weeks (yellow). \par}

\pagebreak









\end{document}